\documentclass[a4,12pt]{amsart}
\usepackage[mathscr]{eucal}
\usepackage{amssymb}
\usepackage{latexsym}
\usepackage{amsthm}
\newtheorem{theorem}{Theorem}[section]

\newtheorem{remark}{Remark}
\newtheorem{proposition}{Proposition}[section]
\newtheorem{lemma}{Lemma}[section]

\makeatletter
\@addtoreset{equation}{section}

\title[a  clamped plate problem]
{A gap for eigenvalues of a  clamped plate problem}
\author{Daguang Chen,  Qing-Ming Cheng and Guoxin Wei}
\address{Daguang Chen\\
Department of Mathematical Sciences,
Tsinghua University,
Beijing 100084, P. R. China, dgchen@math.tsinghua.edu.cn
}
\address{Qing-Ming Cheng \\ Department of Applied Mathematics, Faculty of Sciences,
Fukuoka  University, 814-0180, Fukuoka,  Japan, cheng@fukuoka-u.ac.jp}
\address{Guoxin Wei \\  School of Mathematical Sciences, South China Normal University,
510631, Guangzhou,  P. R. China, weiguoxin@tsinghua.org.cn}

\begin{document}
\maketitle

\begin{abstract}
This paper  studies eigenvalues of the clamped plate problem on a bounded
domain in an $n$-dimensional Euclidean space.  We give an estimate for
the gap between
$\sqrt {\Gamma_{k+1}-\Gamma_{1}}$ and $\sqrt {\Gamma_{k}-\Gamma_{1}}$, for any  positive integer $k$.
According to  the asymptotic  formula of Agmon and Pleijel,
we know, the gap between
$\sqrt {\Gamma_{k+1}-\Gamma_{1}}$ and $\sqrt {\Gamma_{k}-\Gamma_{1}}$ is bounded by a term
with a lower order  $k^{\frac1n}$ in the sense of the asymptotic formula of Agmon and Peijel,
where $\Gamma_j$ denotes the $j^{^{\text{th}}}$ eigenvalue
of  the clamped plate problem.

\end{abstract}

\footnotetext{The first author is supported by NSFC.
The second author is partially  supported by JSPS Grant-in-Aid for Scientific Research (B): No.16H03937.
The third author is supported by NSFC No. 11371150.}

\section{introduction}

\noindent
It is well-known that study on eigenvalues of the eigenvalue problem of  elliptic operators
is a very important subject in geometry and analysis.

\noindent
Let $\Omega$ be a bounded domain with piecewise smooth boundary
in an $n$-dimensional complete Riemannian
manifold $M$.
The following is called {\it the  Dirichlet eigenvalue problem of Laplacian}:
\begin{equation}
\begin{cases}
\Delta u=-\lambda u  &  \text{in $\Omega$}, \\
u=0 & \text{on $\partial \Omega$},
\end{cases}
\end{equation} \\
where $\Delta$ is the Laplacian on $M$.
Many  mathematicians  study universal estimates of  eigenvalues of the Dirichlet  eigenvalue problem of Laplacian.
As main developments for study on universal estimates of eigenvalues,
Payne, P\'olya and Weinberger \cite{PPW}, Hile and Protter \cite{HP},  Yang \cite{Y} makes very important contributions
for bounded domains in Euclidean spaces (see  Ashbaugh \cite{As1, As2, As3}). For domains in sphere, Cheng and Yang  \cite{CY1}
obtains optimal universal estimates on eigenvalues. For bounded domains in complete Riemannian manifolds,
universal estimates on eigenvalues have been obtained by  in
Cheng and Yang  \cite{CY2}, Chen and Cheng \cite{CC}, Chen, Zheng and Yang \cite{CZY}
and El Soufi, Harrell and Ilias \cite{EHI}, Cheng \cite{C1}  and so on. By making use of the universal
estimates on eigenvalues and the recursive inequality of Cheng and Yang \cite{CY4}, Cheng and Yang \cite{CY5} obtain sharp lower
bounds and upper bounds for the $k^{\text{th}}$ eigenvalues of
the Dirichlet  eigenvalue problem of Laplacian in the sense of order of $k$.  For bounded domains in the Euclidean space, by making use of Fourier transform, Li and Yau \cite{LY} gives an optimal lower bound for  the average of the first $k$ eigenvalues of
the Dirichlet  eigenvalue problem of Laplacian.  Recently, in \cite{As3}, Ashbaugh gives a very nice survey for  estimates on eigenvalues of the Dirichlet  eigenvalue problem of Laplacian for  bounded domains in Euclidean space.  For bounded domains in complete Riemannian manifolds, see the very nice book of Urakawa \cite{U}.

\noindent
In this paper,  we consider an eigenvalue problem of the biharmonic operator $\Delta^2$  on a bounded domain
with piecewise smooth boundary in an $n$-dimensional complete Riemannian manifold $M$,
which is also  called {\it the  clamped
plate problem}:
\begin{equation}\label{ccp}
\begin{cases}
\Delta^2 u=\Gamma u  &  \text{in $\Omega$} \\
u=\displaystyle{ \frac{\partial u}{\partial \nu}}=0 & \text{on $\partial \Omega$},
\end{cases}
\end{equation}
where $\Delta^2$ denotes the biharmonic operator on $M$, and $\nu$
is the outward unit normal of $\partial \Omega$.

\noindent
When $\Omega$ is a bounded domain in $\mathbf R^n$,
Agmon and Pleijel
give the following  asymptotic  formula of  eigenvalues of the clamped plate problem (1.2):
$$
\Gamma_k\sim
\dfrac{16\pi^4}{\big(\omega_n\text{vol}(\Omega)\big)^{\frac{4}{n}}}k^{\frac{4}{n}},\
\ \  k\rightarrow\infty.
  $$
This implies that
\begin{equation*}
\frac{1}{k}\sum_{j=1}^k\Gamma_j
\sim\frac{n}{n+4}\dfrac{16\pi^4}{\big(\omega_n\text{vol}(\Omega)\big)^{\frac{4}{n}}}k^{\frac{4}{n}},
\ \ k\rightarrow\infty,
\end{equation*}
where $\Gamma_j$ denotes the $j^{\text{\rm th}}$ eigenvalue of the clamped plate problem \ref{ccp},
 $\text{vol}(\Omega)$ and $\omega_n$ denote  volumes of $\Omega$ and the unit ball in $\mathbf R^n$, respectively.
Furthermore, by making use of the Fourier transform and a lemma due to  H\"ormander,
Levine and Protter  \cite{LP} proves that  eigenvalues of
the clamped plate problem \ref{ccp}  satisfy
$$\dfrac{1}{k}\sum_{j=1}^k\Gamma_j
\geq\frac{n}{n+4}\dfrac{16\pi^4}{\big(\omega_n\text{vol}(\Omega)\big)^{\frac{4}{n}}}k^{\frac{4}{n}}.
$$
The above  formula  shows that the coefficient of $k^{\frac{4}{n}}$ is the best possible constant.
and the order of $k$ is optimal according to the asymptotic formula of Agmon and Peijel.
Cheng and Wei \cite{CW1, CW2}   and Cheng, Qi and Wei \cite{CQW}  generalize the result of Levine and Protter by adding the  lower terms.

\noindent
On the other hand,  it is a very difficult problem to obtain a  sharp estimate for  the upper bound of eigenvalues
with optimal order of $k$
of the clamped plate problem (\ref{ccp}).  For  estimates for upper bounds of eigenvalues
and estimates of two consecutive eigenvalues of the clamped plate problem,
Payne, P\'olya and Weinberger \cite{PPW} proves
\begin{equation}
\Gamma_{k+1} - \Gamma_{k} \leq \frac{8(n+2)}{n^{2} k} \sum_{i=1}^{k} \Gamma_{i}.
\end{equation}
Chen and Qian \cite{CQ} and Hook \cite{H}, independently, extend
 the above inequality to
\begin{equation}
\frac{n^{2} k^{2}}{8(n+2)} \leq
 \sum_{i=1}^{k} \frac{\Gamma_{i}^{\frac{1}{2}}}{\Gamma_{k+1}
 - \Gamma_{i}} \sum_{i=1}^{k} \Gamma_{i}^{\frac{1}{2}}.
\end{equation}
 Cheng and Yang \cite{CY2}  and Wang and Xia \cite{WX} prove
\begin{equation}
\sum_{i=1}^k (\Gamma_{k+1}-\Gamma_{i})^2 \leq \displaystyle{\frac{8(n+2)}{n^2}}\sum_{i=1}^k(\Gamma_{k+1}-\Gamma_{i})\Gamma_{i}
\end{equation}
In the open problem section ( of the 6th International Chinese Congress of Mathematicians,
July 9-14,  2013, Taiwan National University), the second author proposes the following problem:

\vskip2mm
\noindent
{\bf Conjecture 1.1}.  Eigenvalues of the clamped plate problem (\ref{ccp}) for a bounded domain in $\mathbf {R}^n$ satisfies
\begin{equation}
\sum_{i=1}^k (\Gamma_{k+1}-\Gamma_{i})^2 \leq \displaystyle{\frac{8}{n}}\sum_{i=1}^k(\Gamma_{k+1}-\Gamma_{i})\Gamma_{i}.
\end{equation}

\noindent
In fact, if one may prove the conjecture 1.1, by making use of the recursive formula of Cheng and Yang \cite{CY4},
one may obtain the sharp estimates on the upper bound  of the $k^{\text{th}}$ eigenvalue, in the sense of the order of $k$,  of the clamped plate problem.

\noindent
In this paper, we study the gap of two consecutive eigenvalues of the clamped plate problem. We obtain the following:
\begin{theorem}
Let $\Omega$  be a bounded domain in the Euclidean space $ \mathbf {R}^n $. Then, for any integer $k\geq 0$, we have
\begin{equation}
(\sqrt {\Gamma_{k+1}-\Gamma_{1}} -\sqrt {\Gamma_{k}-\Gamma_{1}} )^2
\leq  \dfrac{16\sqrt{\Gamma_1}}{n}\bigl\{(\Gamma_{k+1}-\Gamma_1)(\Gamma_{k}-\Gamma_1)\bigl\}^{\frac14}+C,
\end{equation}
where
$$
C=\max\biggl\{\dfrac{8\int_{\Omega}|\nabla\Delta u_1|^2dv}{(n+2)\|\nabla u_1\|^2 }, \
\dfrac{4(n+12)\Gamma_1+16\int_{\Omega}\sum_{m=1}^n(\dfrac{\partial^2 u_1}{\partial x_m^2})^2dv}{n}\biggl\}.
$$
is constant only depending on the  dimension $n$,  the first eigenvalue $\Gamma_1$ and the normalized first eigenfunction $u_1$.
\end{theorem}

\begin{remark}  According to  the asymptotic  formula of Agmon and Pleijel,
we have
$$
\lim_{k\to \infty}\dfrac{\Gamma_k}{k^{\frac{4}{n}}}=
\dfrac{16\pi^4}{\big(\omega_n\text{\rm vol}(\Omega)\big)^{\frac{4}{n}}}.
  $$
From our theorem,
we know, the gap between
$\sqrt {\Gamma_{k+1}-\Gamma_{1}}$ and $\sqrt {\Gamma_{k}-\Gamma_{1}}$ is bounded by a term with a lower order  $k^{\frac1n}$
in the sense of the asymptotic formula of Agmon and Peijel.

\vskip1mm
\noindent
Since
$$
\Gamma_{k+1}-\Gamma_{k}=
(\sqrt {\Gamma_{k+1}-\Gamma_{1}}-\sqrt {\Gamma_{k}-\Gamma_{1}})
(\sqrt {\Gamma_{k+1}-\Gamma_{1}}+\sqrt {\Gamma_{k}-\Gamma_{1}}),
$$
according to the asymptotic  formula of Agmon and Pleijel, we know that the gap between  $\Gamma_{k+1}$
and $\Gamma_{k}$ is bounded by a term with a lower order  $k^{\frac3n}$.
\end{remark}

\vskip5mm
\section{A general result}

\vskip 6mm
\noindent
Let $\Omega$ be a bounded domain with piecewise smooth boundary in an $n$-dimensional
complete Riemannian manifold $M$.
Let $u_i$ be  an eigenfunction corresponding to  the eigenvalue $\Gamma_{i}$ such that
\begin{equation}
\begin{cases}
\Delta^2{u_i}=\Gamma_i u_i & \text{in $\Omega$} \\
u_i=\displaystyle{\frac{\partial u_i}{\partial\nu}}=0 & \text{on $\partial \Omega$} \\
\displaystyle{\int_\Omega u_i u_jdv}=\delta_{ij},  \  i, j=1, 2, \cdots,
\end{cases}
\end{equation}
where eigenvalues are accounted according to their multiplicities.
Thus, we know that $\{u_j\}_{j=1}^{\infty}$ forms an  orthonormal base of $L^2(\Omega)$-function space.
For any smooth function $g$, we can write
$$
gu_1=\sum_{j=1}^{\infty}r_{j}u_j, \quad  \|gu_1\|^2=\int_{\Omega}(gu_1)^2dv=\sum_{j=1}^{\infty}r_{j}^2,
$$
where $r_{j}=\displaystyle{\int_\Omega g u_1 u_jdv}$, for $j=1, 2, \cdots$.
For any positive integer $k$, we define
\begin{equation}
\varphi:=gu_1-\sum_{j=1}^k r_{j} u_j.
\end{equation}
By a simple calculation, we obtain
\begin{equation}
\displaystyle{\int_\Omega u_j\varphi dv=0},   \ \  j=1,\cdots,k.
\end{equation}
Hence
$$
\|\varphi\|^2=\sum_{j=k+1}^{\infty}r_{j}^2.
$$
Defining
\begin{equation}\label{p}
p=\Delta^2 g \cdot u_1
+2\nabla(\Delta g)\cdot \nabla u_1+2\Delta g\Delta u_1
+2\Delta(\nabla g \cdot \nabla u_1)
+2\nabla g \cdot \nabla(\Delta u_1),
\end{equation}
we have
$$
p=\sum_{j=1}^{\infty}s_{j}u_j, \quad \|p\|^2=\sum_{j=1}^{\infty}s_{j}^2,
$$
where
$$
 s_{j}=\displaystyle{\int_{\Omega}}p u_jdv.
$$
Since
\begin{align*}
& 2\displaystyle{\int_{\Omega}}(\Delta u_j \nabla g \cdot \nabla u_1 - \Delta u_1 \nabla g \cdot \nabla u_j)dv \\
& =(\Gamma_j - \Gamma_1)r_{j} - \displaystyle{\int_{\Omega}}u_1\Delta u_j \Delta gdv
+\displaystyle{\int_{\Omega}}u_j\Delta u_1\Delta gdv,
\end{align*} \\
we can infer
\begin{equation}
s_{j} =( \Gamma_j - \Gamma_1 ) r_{j}.
\end{equation} \\
Thus, we get
$$
\|p\|^2=\sum_{j=1}^{\infty}(\Gamma_j-\Gamma_1)^2r_{j}^2.
$$
$$
\int_{\Omega}gu_1pdv=\int_{\Omega}gu_1\sum_{j=1}^{\infty}s_{j}u_jdv=\sum_{j=1}^{\infty}s_{j}r_{j}
=\sum_{j=1}^{\infty}(\Gamma_{j}-\Gamma_1)r_{j}^2.
$$
From the definition of $\varphi$, we have
$$
\int_{\Omega}\varphi p dv=\int_{\Omega}(gu_1-\sum_{j=1}^k r_{j}u_j)p dv
=\sum_{j=1}^{\infty}(\Gamma_{j}-\Gamma_1)r_{j}^2-\sum_{j=1}^{k}(\Gamma_{j}-\Gamma_1)r_{j}^2.
$$
Hence, we obtain
$$
\int_{\Omega}\varphi pdv=
\sum_{j=k+1}^{\infty}(\Gamma_{j}-\Gamma_1)r_{j}^2.
$$

\noindent
The following  algebraic lemma  plays an important role in this paper,
which  may  be found in  Chen-Yang-Zheng \cite{CZY}, essentially.
For reader's convenient, we give a detailed proof of it in the Appendix.

\begin{lemma}\label{lemma1}
Let $\{\mu_j\}_{j=k+1}^{\infty}$ be a sequence satisfying
$$
0\leq \mu_{k+1}\leq \mu_{k+2}\leq \cdots \to \infty.
$$
If  a sequence $\{a_j\}_{j=k+1}^{\infty}$ satisfies $\sum_{j=k+1}^{\infty}\mu_j^2a_j^2=A<\infty $ and
$\sum_{j=k+1}^{\infty}a_j^2=B<\infty $, then we have
$$
\sum_{j=k+1}^{\infty}\mu_ja_j^2\leq \dfrac{A+\mu_{k+1}\mu_{k+2}B}{\mu_{k+1}+\mu_{k+2}}.
$$
\end{lemma}

\noindent
By   applying  the lemma 2.1 with $\mu_j=\Gamma_j-\Gamma_1$ and $a_j=r_{j}$, we obtain
\begin{equation*}
\begin{aligned}
&\bigl\{(\Gamma_{k+1}-\Gamma_1)+(\Gamma_{k+2}-\Gamma_1)\bigl\}\int_{\Omega}\varphi pdv\\
&\leq  \bigl(\|p\|^2-\sum_{j=1}^{k}(\Gamma_{j}-\Gamma_1)^2r_{j}^2\bigl)
+(\Gamma_{k+1}-\Gamma_1)(\Gamma_{k+2}-\Gamma_1)\|\varphi\|^2,
\end{aligned}
\end{equation*}
namely,
\begin{equation*}
\begin{aligned}
&\bigl\{(\Gamma_{k+1}-\Gamma_1)+(\Gamma_{k+2}-\Gamma_1)\bigl\}(\int_{\Omega}gu_1pdv-\sum_{j=1}^{k}(\Gamma_{j}-\Gamma_1)r_{j}^2)\\
&\leq  \bigl(\|p\|^2-\sum_{j=1}^{k}(\Gamma_{j}-\Gamma_1)^2r_{j}^2\bigl)
+(\Gamma_{k+1}-\Gamma_1)(\Gamma_{k+2}-\Gamma_1)(\|gu_1\|^2-\sum_{j=1}^{k}r_{j}^2).
\end{aligned}
\end{equation*}
Since
\begin{equation*}
\begin{aligned}
&\bigl\{(\Gamma_{k+1}-\Gamma_1)+(\Gamma_{k+2}-\Gamma_1)\bigl\}\sum_{j=1}^{k}(\Gamma_{j}-\Gamma_1)r_{j}^2\\
&\leq \sum_{j=1}^{k}(\Gamma_{j}-\Gamma_1)^2r_{j}^2
+(\Gamma_{k+1}-\Gamma_1)(\Gamma_{k+2}-\Gamma_1)\sum_{j=1}^{k}r_{j}^2,
\end{aligned}
\end{equation*}
we have
\begin{equation*}
\begin{aligned}
&\bigl\{(\Gamma_{k+1}-\Gamma_1)+(\Gamma_{k+2}-\Gamma_1)\bigl\}\int_{\Omega}gu_1pdv
\leq  \|p\|^2
+(\Gamma_{k+1}-\Gamma_1)(\Gamma_{k+2}-\Gamma_1)\|gu_1\|^2.
\end{aligned}
\end{equation*}
Thus, we have proved the following:
\begin{theorem}
Let $\Omega$ be a bounded domain in  an $n$-dimensional complete Riemannian manifold $M$.
Assume that  $\Gamma_{i}$ is the $i^{\text{th}}$ eigenvalue of the  clamped plate problem {\rm (\ref{ccp})}.
For any  smooth function $g$, we have, for any  integer $k$,
\begin{equation*}
\begin{aligned}
&\bigl\{(\Gamma_{k+2}-\Gamma_1)+(\Gamma_{k+1}-\Gamma_1)\bigl\}\int_{\Omega}gu_1pdv
\leq  \|p\|^2
+(\Gamma_{k+2}-\Gamma_1)(\Gamma_{k+1}-\Gamma_1)\|gu_1\|^2,
\end{aligned}
\end{equation*}
where $p$ is defined by the formula {\rm ($\ref{p}$)} and $u_1$ is the normalized first eigenfunction corresponding to
the first  eigenvalue $\Gamma_1$.
 \end{theorem}

\begin{lemma}
\begin{equation*}
\int_{\Omega}gu_1pdv=\displaystyle{\int_{\Omega}}\biggl\{(\Delta g)^2 u_1^2
+4(\nabla g \cdot \nabla u_1)^2 -2|\nabla g |^2 u_1 \Delta u_1
+4u_1\Delta g \nabla g\cdot \nabla u_1\biggl\}dv.
\end{equation*}
\end{lemma}
\begin{proof}
From  Stokes' theorem,  we infer
\begin{equation*}
2\displaystyle{\int_{\Omega}}g u_1\nabla(\Delta g) \cdot\nabla u_1dv
=\displaystyle{\int_{\Omega}}\biggl\{2u_1\Delta g\nabla u_1\cdot  \nabla g
+u_1^2(\Delta g)^2 -g u_1^2\Delta^2 g\biggl\}dv,
\end{equation*}
\begin{equation*}
2\displaystyle{\int_{\Omega}} g u_1\Delta(\nabla g \cdot \nabla u_1)dv
=\displaystyle{\int_{\Omega}}\biggl\{2u_1\Delta g \nabla g \cdot \nabla u_1
+4(\nabla g \cdot \nabla u_1)^2
+2g \Delta u_1\nabla g \cdot \nabla u_1\biggl\}dv,
\end{equation*}
\begin{equation*}
2\displaystyle{\int_{\Omega}}gu_1\nabla g \cdot \nabla (\Delta u_1)dv
=-2\displaystyle{\int_{\Omega}}\biggl(|\nabla g |^2 u_1 \Delta u_1
+g \Delta u_1\nabla g \cdot \nabla u_1
+ g \Delta g u_1 \Delta u_1\biggl)dv.
\end{equation*}\\
From the definition of $p$, we obtain
\begin{equation*}
\int_{\Omega}gu_1pdv=\displaystyle{\int_{\Omega}}\biggl\{(\Delta g)^2 u_1^2
+4(\nabla g \cdot \nabla u_1)^2 -2|\nabla g |^2 u_1 \Delta u_1
+4u_1 \Delta g \nabla g\cdot \nabla u_1\biggl\}dv.
\end{equation*}

\end{proof}

\vskip2mm
\noindent
For any smooth function $f$ in $M$ and constant $a$, we consider $g_1=\cos (af)$. We have
\begin{equation*}
\nabla g_1=-a\sin (af)\nabla f, \quad \Delta g_1 =-a^2\cos (af)|\nabla f|^2-a\sin (af) \Delta f
\end{equation*}

\begin{equation*}
 \begin{aligned}
 \nabla \Delta g_1 &=a^3\sin (af)|\nabla f|^2\nabla f-a^2\cos (af)\nabla(|\nabla f|^2)\\
 &-a^2\cos (af) \Delta f\nabla f
-a\sin (af)\nabla(\Delta f)
\end{aligned}
\end{equation*}

\begin{equation*}
\begin{aligned}
 \Delta^2 g_1 &=a^4\cos (af)|\nabla f|^4+2a^3\sin (af) \nabla(|\nabla f|^2)\cdot\nabla f+2a^3\sin (af) |\nabla f|^2\Delta f\\
 & - a^2\cos (af)\Delta (|\nabla f|^2) -2a^2\cos (af) \nabla(\Delta f)\cdot\nabla f\\
 &-a^2\cos (af) (\Delta f)^2-a\sin (af) \Delta^2f.
\end{aligned}
\end{equation*}
In the same way, for  $g_2=\sin  (af)$, we have
\begin{equation*}
\nabla g_2=a\cos (af)\nabla f, \quad \Delta g_2 =-a^2\sin (af)|\nabla f|^2+a\cos (af) \Delta f
\end{equation*}

\begin{equation*}
\begin{aligned}
 \nabla \Delta g_2 &=-a^3\cos (af)|\nabla f|^2\nabla f-a^2\sin (af)\nabla(|\nabla f|^2)\\
 &-a^2\sin (af) \Delta f\nabla f
+a\cos (af)\nabla(\Delta f)
\end{aligned}
\end{equation*}

\begin{equation*}
\begin{aligned}
 \Delta^2 g_2 &=a^4\sin (af)|\nabla f|^4-2a^3\cos (af) \nabla(|\nabla f|^2)\cdot\nabla f\\
 &-2a^3\cos (af) |\nabla f|^2\Delta f- a^2\sin (af)\Delta (|\nabla f|^2) \\
 &-2a^2\sin (af) \nabla(\Delta f)\cdot\nabla f-a^2\sin (af) (\Delta f)^2+a\cos (af) \Delta^2f.
\end{aligned}
\end{equation*}
Thus, we obtain the following:
\begin{lemma}
If the function $f$ satisfies $|\nabla f|^2=1$ and $\Delta f =b=$constant, we have
\begin{equation*}
\nabla g_1=-a\sin (af)\nabla f, \quad \Delta g_1 =-a^2\cos (af)-ab\sin (af),
\end{equation*}
\begin{equation*}
 \begin{aligned}
 \nabla \Delta g_1 &=a^3\sin (af)\nabla f-a^2b\cos (af) \nabla f,
\end{aligned}
\end{equation*}
\begin{equation*}
\begin{aligned}
 \Delta^2 g_1 &=a^4\cos (af)+2a^3b\sin(af)-a^2b^2\cos (af),
\end{aligned}
\end{equation*}
\begin{equation*}
\nabla g_2=a\cos (af)\nabla f, \quad \Delta g_2 =-a^2\sin (af)+ab\cos (af),
\end{equation*}

\begin{equation*}
\begin{aligned}
 \nabla \Delta g_2 &=-a^3\cos (af)\nabla f-a^2b\sin (af) \nabla f,
\end{aligned}
\end{equation*}
\begin{equation*}
\begin{aligned}
 \Delta^2 g_2 &=a^4\sin (af) -2a^3b\cos (af)-a^2b^2\sin (af).
\end{aligned}
\end{equation*}
\end{lemma}

\noindent
By defining
\begin{equation}
p_1=\Delta^2 g_1 \cdot u_1
+2\nabla(\Delta g_1)\cdot \nabla u_1+2\Delta g_1\Delta u_1
+2\Delta(\nabla g_1 \cdot \nabla u_1)
+2\nabla g_1 \cdot \nabla(\Delta u_1),
\end{equation}
and
\begin{equation}
p_2=\Delta^2 g_2 \cdot u_1
+2\nabla(\Delta g_2)\cdot \nabla u_1+2\Delta g_2\Delta u_1
+2\Delta(\nabla g_2 \cdot \nabla u_1)
+2\nabla g_2 \cdot \nabla(\Delta u_1),
\end{equation}
we have
\begin{proposition}
If the function $f$ satisfies $|\nabla f|^2=1$ and $\Delta f =b=$constant, we have
\begin{equation*}
\begin{aligned}
&|p_1|^2+|p_2|^2\\
&=\biggl((a^4-a^2b^2\bigl)u_1-4a^2b\nabla f\cdot \nabla u_1-2a^2\Delta u_1
-4a^2\nabla f\cdot \nabla(\nabla f \cdot \nabla u_1)\biggl)^2\\
&+\biggl(2a^3bu_1+4a^3\nabla f\cdot \nabla u_1-2ab\Delta u_1
-2a\Delta(\nabla f \cdot \nabla u_1)
-2a\nabla f \cdot \nabla(\Delta u_1)\biggl )^2.
\end{aligned}
\end{equation*}
\end{proposition}
\begin{proof}
From the above lemma 2.3, we have
\begin{equation*}
\begin{aligned}
p_1&=\bigl(a^4\cos (af)+2a^3b\sin(af)-a^2b^2\cos (af)\bigl)u_1\\
&+2\bigl(a^3\sin (af)-a^2b\cos (af) \bigl)\nabla f\cdot \nabla u_1-2(a^2\cos (af)+ab\sin (af))\Delta u_1\\
&-2a\Delta(\sin (af)\nabla f \cdot \nabla u_1)
-2a\sin (af)\nabla f \cdot \nabla(\Delta u_1)
\end{aligned}
\end{equation*}
and
\begin{equation*}
\begin{aligned}
&\Delta(\sin (af)\nabla f \cdot \nabla u_1)=\sin (af)\Delta(\nabla f \cdot \nabla u_1)\\
&+2a\cos (af)\nabla f\cdot \nabla(\nabla f \cdot \nabla u_1))
-\biggl(a^2\sin (af)-ab\cos (af)\biggl)\nabla f \cdot \nabla u_1,
\end{aligned}
\end{equation*}
\begin{equation*}
\begin{aligned}
p_2&=\bigl(a^4\sin (af) -2a^3b\cos (af)-a^2b^2\sin (af)\bigl)u_1\\
&-2\bigl(a^3\cos (af)+a^2b\sin (af)  \bigl)\nabla f\cdot \nabla u_1-2(a^2\sin (af)-ab\cos (af))\Delta u_1\\
&+2a\Delta(\cos (af)\nabla f \cdot \nabla u_1)
+2a\cos (af)\nabla f \cdot \nabla(\Delta u_1)
\end{aligned}
\end{equation*}
and
\begin{equation*}
\begin{aligned}
&\Delta(\cos (af)\nabla f \cdot \nabla u_1)=\cos (af)\Delta(\nabla f \cdot \nabla u_1)\\
&-2a\sin (af)\nabla f\cdot \nabla(\nabla f \cdot \nabla u_1)
-\biggl(a^2\cos (af)+ab\sin(af)\biggl)\nabla f \cdot \nabla u_1.
\end{aligned}
\end{equation*}
Hence,   we infer
\begin{equation*}
\begin{aligned}
&p_1=\biggl((a^4-a^2b^2)u_1-4a^2b \nabla f\cdot \nabla u_1-2a^2\Delta u_1
-4a^2\nabla f\cdot \nabla(\nabla f \cdot \nabla u_1)\biggl)\cos (af)\\
&+\biggl(2a^3bu_1+4a^3\nabla f\cdot \nabla u_1-2ab\Delta u_1
-2a\Delta(\nabla f \cdot \nabla u_1)
-2a\nabla f \cdot \nabla(\Delta u_1)\biggl)\sin (af),
\end{aligned}
\end{equation*}
\begin{equation*}
\begin{aligned}
&p_2
=\biggl((a^4-a^2b^2\bigl)u_1-4a^2b\nabla f\cdot \nabla u_1-2a^2\Delta u_1
-4a^2\nabla f\cdot \nabla(\nabla f \cdot \nabla u_1)\biggl)\sin (af)\\
&-\biggl(2a^3bu_1+4a^3\nabla f\cdot \nabla u_1-2ab\Delta u_1
-2a\Delta(\nabla f \cdot \nabla u_1)
-2a\nabla f \cdot \nabla(\Delta u_1)\biggl )\cos (af).
\end{aligned}
\end{equation*}
From the above two equalities, we obtain
\begin{equation*}
\begin{aligned}
&|p_1|^2+|p_2|^2\\
&=\biggl((a^4-a^2b^2\bigl)u_1-4a^2b\nabla f\cdot \nabla u_1-2a^2\Delta u_1
-4a^2\nabla f\cdot \nabla(\nabla f \cdot \nabla u_1)\biggl)^2\\
&+\biggl(2a^3bu_1+4a^3\nabla f\cdot \nabla u_1-2ab\Delta u_1
-2a\Delta(\nabla f \cdot \nabla u_1)
-2a\nabla f \cdot \nabla(\Delta u_1)\biggl )^2.
\end{aligned}
\end{equation*}

\end{proof}

\begin{proposition}
If the function $f$ satisfies $|\nabla f|^2=1$ and $\Delta f =b=$constant, we have
\begin{equation*}
\begin{aligned}
\int_{\Omega}g_1u_1p_1dv+\int_{\Omega}g_2u_1p_2dv&=\displaystyle{\int_{\Omega}}\biggl\{(a^4-a^2b^2) u_1^2
+4a^2(\nabla f \cdot \nabla u_1)^2 -2a^2u_1 \Delta u_1\biggl\}dv.
\end{aligned}
\end{equation*}
\end{proposition}

\begin{proof} Since
\begin{equation*}
\begin{aligned}
\int_{\Omega}g_1u_1p_1dv&=\displaystyle{\int_{\Omega}}\biggl\{(a^2\cos (af)+ab\sin (af))^2 u_1^2 \\
&+4a^2(\sin (af))^2(\nabla f \cdot \nabla u_1)^2 -2a^2(\sin (af))^2 u_1 \Delta u_1\\
&+4a\sin (af) (a^2\cos (af)+ab\sin (af))u_1 \nabla f\cdot \nabla u_1\biggl\}dv
\end{aligned}
\end{equation*}
and
\begin{equation*}
\begin{aligned}
\int_{\Omega}g_2u_1p_2dv&=\displaystyle{\int_{\Omega}}\biggl\{(a^2\sin (af)-ab\cos (af))^2 u_1^2 \\
&+4a^2(\cos (af))^2(\nabla f \cdot \nabla u_1)^2 -2a^2(\cos (af))^2 u_1 \Delta u_1\\
&-4a\cos (af)(a^2\sin (af)-ab\cos (af))u_1\nabla f\cdot \nabla u_1\biggl\}dv,
\end{aligned}
\end{equation*}
we infer
\begin{equation*}
\begin{aligned}
&\int_{\Omega}g_1u_1p_1dv+\int_{\Omega}g_2u_1p_2dv\\
&=\displaystyle{\int_{\Omega}}\biggl\{(a^4+a^2b^2) u_1^2
+4a^2(\nabla f \cdot \nabla u_1)^2
-2a^2u_1 \Delta u_1+4a^2bu_1 \nabla f\cdot \nabla u_1\biggl\}dv.
\end{aligned}
\end{equation*}
According to Stokes formula, we know
$$
\int_{\Omega}2u_1\nabla f\cdot \nabla u_1dv=-\int_{\Omega}bu^2_1dv.
$$
Hence,
we get
\begin{equation*}
\begin{aligned}
\int_{\Omega}g_1u_1p_1dv+\int_{\Omega}g_2u_1p_2dv&=\displaystyle{\int_{\Omega}}\biggl\{(a^4-a^2b^2) u_1^2
+4a^2(\nabla f \cdot \nabla u_1)^2 -2a^2u_1 \Delta u_1\biggl\}dv.
\end{aligned}
\end{equation*}
\end{proof}

\section{The proof of the theorem 1.1}
\vskip2mm
\noindent
{\it Proof of Theorem} 1.1.  Since $\Omega$ is a bounded domain in the Euclidean  space $\mathbf {R}^n$. Let $(x_1, x_2, \cdots, x_n)$ be
the standard coordinate. By taking $f=x_m$, for $m=1, 2, \cdots, n$, we know
$$
|\nabla f|^2=1, \quad \Delta f=0.
$$
Thus, from the propositions 2.1, we obtain, for $m=1, 2, \cdots, n$,

\begin{equation*}
\begin{aligned}
|p_1|^2+|p_2|^2
&=\biggl(a^4u_1-2a^2\Delta u_1
-4a^2\dfrac{\partial^2 u_1}{\partial x_m^2}
\biggl)^2\\
&+\biggl(4a^3\dfrac{\partial u_1}{\partial x_m}-2a\Delta(\dfrac{\partial u_1}{\partial x_m})
-2a\dfrac{\partial(\Delta  u_1)}{\partial x_m}\biggl )^2\\
&=a^4\biggl(a^2u_1-2\Delta u_1
-4\dfrac{\partial^2 u_1}{\partial x_m^2}
\biggl)^2+16a^2\biggl(a^2\dfrac{\partial u_1}{\partial x_m}-\Delta(\dfrac{\partial u_1}{\partial x_m})\biggl )^2.
\end{aligned}
\end{equation*}
Hence,
\begin{equation*}
\begin{aligned}
\int_{\Omega}|p_1|^2dv+\int_{\Omega}|p_2|^2dv
&=\int_{\Omega}a^4\biggl(a^2u_1-2\Delta u_1
-4\dfrac{\partial^2 u_1}{\partial x_m^2}
\biggl)^2dv\\
&+\int_{\Omega}16a^2\biggl(a^2\dfrac{\partial u_1}{\partial x_m}-\Delta(\dfrac{\partial u_1}{\partial x_m})
\biggl )^2dv
\end{aligned}
\end{equation*}
holds.
By a direct computation, we infer
\begin{equation*}
\begin{aligned}
&\int_{\Omega}\biggl(a^2u_1-2\Delta u_1
-4\dfrac{\partial^2 u_1}{\partial x_m^2}\biggl)^2dv
\\&=a^4+4\Gamma_1+16\int_{\Omega}(\dfrac{\partial^2 u_1}{\partial x_m^2})^2dv
\\
&+4a^2\int_{\Omega}|\nabla u_1|^2dv-8a^2\int_{\Omega}u_1\dfrac{\partial^2 u_1}{\partial x_m^2}dv
+16\int_{\Omega}\Delta u_1\dfrac{\partial^2 u_1}{\partial x_m^2}dv
\end{aligned}
\end{equation*}
and
\begin{equation*}
\begin{aligned}
&\int_{\Omega}\biggl(a^2\dfrac{\partial u_1}{\partial x_m}-\Delta(\dfrac{\partial u_1}{\partial x_m})
\biggl )^2dv\\
&=\int_{\Omega}\biggl(a^4(\dfrac{\partial u_1}{\partial x_m})^2+(\Delta(\dfrac{\partial u_1}{\partial x_m}))^2
-2a^2\dfrac{\partial u_1}{\partial x_m}\Delta(\dfrac{\partial u_1}{\partial x_m})\biggl )dv.\\
\end{aligned}
\end{equation*}
We derive
\begin{equation*}
\begin{aligned}
&\int_{\Omega}|p_1|^2dv+\int_{\Omega}|p_2|^2dv\\
&=a^4\biggl\{a^4+4\Gamma_1+16\int_{\Omega}(\dfrac{\partial^2 u_1}{\partial x_m^2})^2dv
+4a^2\int_{\Omega}|\nabla u_1|^2dv\\
&-8a^2\int_{\Omega}u_1\dfrac{\partial^2 u_1}{\partial x_m^2}dv
+16\int_{\Omega}\Delta u_1\dfrac{\partial^2 u_1}{\partial x_m^2}dv\biggl\}\\
&+16a^2\biggl\{\int_{\Omega}\biggl(a^4(\dfrac{\partial u_1}{\partial x_m})^2+(\Delta(\dfrac{\partial u_1}{\partial x_m}))^2
-2a^2\dfrac{\partial u_1}{\partial x_m}\Delta(\dfrac{\partial u_1}{\partial x_m})\biggl )dv\biggl\}.
\end{aligned}
\end{equation*}
From the proposition 2.2, we infer

\begin{equation*}
\begin{aligned}
\int_{\Omega}g_1u_1p_1dv+\int_{\Omega}g_2u_1p_2dv&=\displaystyle{\int_{\Omega}}\biggl\{a^4 u_1^2
+4a^2(\dfrac{\partial u_1}{\partial x_m})^2 -2a^2u_1\Delta u_1\biggl\}dv\\
&=a^4
+2a^2\displaystyle{\int_{\Omega}}\biggl\{2(\dfrac{\partial u_1}{\partial x_m})^2 +|\nabla u_1|^2 \biggl\}dv.
\end{aligned}
\end{equation*}
We apply the theorem 2.1 to functions $g=g_1$ and $g=g_2$, respectively and take summation for them, we have

\begin{equation}
\begin{aligned}
&\bigl\{(\Gamma_{k+1}-\Gamma_1)+(\Gamma_{k+2}-\Gamma_1)\bigl\}\biggl(a^4
+2a^2\displaystyle{\int_{\Omega}}\biggl\{2(\dfrac{\partial u_1}{\partial x_m})^2 +|\nabla u_1|^2 \biggl\}dv\biggl)\\
&\leq  a^4\biggl\{a^4+4\Gamma_1+16\int_{\Omega}(\dfrac{\partial^2 u_1}{\partial x_m^2})^2dv
+4a^2\int_{\Omega}|\nabla u_1|^2dv\\
&-8a^2\int_{\Omega}u_1\dfrac{\partial^2 u_1}{\partial x_m^2}dv
+16\int_{\Omega}\Delta u_1\dfrac{\partial^2 u_1}{\partial x_m^2}dv\biggl\}\\
&+16a^2\biggl\{\int_{\Omega}\biggl(a^4(\dfrac{\partial u_1}{\partial x_m})^2+(\Delta(\dfrac{\partial u_1}{\partial x_m}))^2
-2a^2\dfrac{\partial u_1}{\partial x_m}\Delta(\dfrac{\partial u_1}{\partial x_m})\biggl )dv\biggl\}\\
&+(\Gamma_{k+1}-\Gamma_1)(\Gamma_{k+2}-\Gamma_1),
\end{aligned}
\end{equation}
Taking summation for $m$ from $1$ to $n$ and making use of Stokes formula, we have
\begin{equation}
\begin{aligned}
&\bigl\{(\Gamma_{k+1}-\Gamma_1)+(\Gamma_{k+2}-\Gamma_1)\bigl\}\bigl(na^4
+2a^2(2+n)\|\nabla u_1\|^2\bigl )\\
&\leq  a^4\biggl\{na^4+4(n+4)\Gamma_1
+4a^2(n+2)\|\nabla u_1\|^2+16\int_{\Omega}\sum_{m=1}^n(\dfrac{\partial^2 u_1}{\partial x_m^2})^2dv\biggl\}\\
&+16a^2\biggl\{a^4\|\nabla u_1\|^2+2a^2\Gamma_1+\int_{\Omega}\sum_{m=1}^n(\Delta(\dfrac{\partial u_1}{\partial x_m}))^2dv\biggl\}\\
&+n(\Gamma_{k+1}-\Gamma_1)(\Gamma_{k+2}-\Gamma_1),
\end{aligned}
\end{equation}
that is,
\begin{equation*}
\begin{aligned}
&\bigl\{(\Gamma_{k+1}-\Gamma_1)+(\Gamma_{k+2}-\Gamma_1)\bigl\}(na^2
+2(n+2)\|\nabla u_1\|^2 )\\
&\leq  a^2\biggl\{na^4
+4a^2(n+6)\|\nabla u_1\|^2+4(n+12)\Gamma_1+16\int_{\Omega}\sum_{m=1}^n(\dfrac{\partial^2 u_1}{\partial x_m^2})^2dv\biggl\}\\
&+16\int_{\Omega}\sum_{m=1}^n(\Delta(\dfrac{\partial u_1}{\partial x_m}))^2dv
+\dfrac{n}{a^2}(\Gamma_{k+1}-\Gamma_1)(\Gamma_{k+2}-\Gamma_1).\\
\end{aligned}
\end{equation*}

Thus,  we obtain
\begin{equation*}
\begin{aligned}
&\bigl\{(\Gamma_{k+1}-\Gamma_1)+(\Gamma_{k+2}-\Gamma_1)\bigl\}\\
&\leq
a^2( a^2+2\frac{n+2}{n}\|\nabla u_1\|^2)
+\dfrac{(\Gamma_{k+1}-\Gamma_1)(\Gamma_{k+2}-\Gamma_1)}{a^2(a^2+2\frac{n+2}{n}\|\nabla u_1\|^2 )}\\
&+\dfrac{16a^4\|\nabla u_1\|^2}{n(a^2+2\frac{n+2}{n}\|\nabla u_1\|^2 )}\\
&+\dfrac{16\int_{\Omega}|\nabla\Delta u_1|^2
+\biggl(4(n+12)\Gamma_1
+16\int_{\Omega}\sum_{m=1}^n(\dfrac{\partial^2 u_1}{\partial x_m^2})^2dv\biggl)a^2}{na^2+2(n+2)\|\nabla u_1\|^2 }.\\
\end{aligned}
\end{equation*}
For $k_1\geq 0$, $k_2>0$ and $k_3>0$, the function $f(t)=\dfrac{k_1+tk_2}{nt+k_3}$, for $ t\geq 0$, satisfies
$$
f(t)\leq \max\{\dfrac{k_1}{k_3}, \ \dfrac{k_2}n\}.
$$
Thus, we have
\begin{equation*}
\begin{aligned}
&\dfrac{16\int_{\Omega}|\nabla\Delta u_1|^2
+\biggl(4(n+12)\Gamma_1
+16\int_{\Omega}\sum_{m=1}^n(\dfrac{\partial^2 u_1}{\partial x_m^2})^2dv\biggl)a^2}{na^2+2(n+2)\|\nabla u_1\|^2 }
\leq C,\\
\end{aligned}
\end{equation*}
where $C$ is given by
$$
C=\max\biggl\{\dfrac{8\int_{\Omega}|\nabla\Delta u_1|^2dv}{(n+2)\|\nabla u_1\|^2 },
\dfrac{4(n+12)\Gamma_1+16\int_{\Omega}\sum_{m=1}^n(\dfrac{\partial^2 u_1}{\partial x_m^2})^2dv}{n}\biggl\}.
$$
If we put
$$
a^2(a^2+2\dfrac{n+2}n\|\nabla u_1\|^2 )=\sqrt{(\Gamma_{k+1}-\Gamma_1)(\Gamma_{k+2}-\Gamma_1)},
$$
we obtain
$$
a^4\leq \sqrt{(\Gamma_{k+1}-\Gamma_1)(\Gamma_{k+2}-\Gamma_1)},
$$
\begin{equation*}
\begin{aligned}
\bigl(\sqrt{\Gamma_{k+2}-\Gamma_1}-\sqrt{\Gamma_{k+1}-\Gamma_1}\bigl)^2\leq \dfrac{16\sqrt{\Gamma_1}}{n}\biggl\{(\Gamma_{k+1}-\Gamma_1)(\Gamma_{k+2}-\Gamma_1)\biggl\}^{\frac14}+C,
\end{aligned}
\end{equation*}
because
\begin{equation*}
\begin{aligned}
\|\nabla u_1\|^2\leq \sqrt{\Gamma_1}.
\end{aligned}
\end{equation*}
If we change $k+2$ and $k+1$ into $k+1$ and $k$, respectively, we know that  the theorem 1.1 is proved.

\section{Appendix}

\noindent
In this Appendix, we shall give a proof of the lemma 2.1.

\vskip2mm
\noindent
{\bf Lemma 2.1.}
Let $\{\mu_j\}_{j=k+1}^{\infty}$ be a sequence satisfying
$$
0\leq \mu_{k+1}\leq \mu_{k+2}\leq \cdots \to \infty.
$$
If  a sequence $\{a_j\}_{j=k+1}^{\infty}$ satisfies $\sum_{j=k+1}^{\infty}\mu_j^2a_j^2=A<\infty $ and
$\sum_{j=k+1}^{\infty}a_j^2=B<\infty $, then we have
$$
\sum_{j=k+1}^{\infty}\mu_ja_j^2\leq \dfrac{A+\mu_{k+1}\mu_{k+2}B}{\mu_{k+1}+\mu_{k+2}}.
$$

\vskip2mm
\begin{proof}
From the Cauchy-Schwarz inequality, we know
$$
\mu_{k+1}\sum_{j=k+1}^{\infty}a_j^2\leq \sum_{j=k+1}^{\infty}\mu_ja_j^2\leq
\sqrt{\sum_{j=k+1}^{\infty}\mu_j^2a_j^2\sum_{j=k+1}^{\infty}a_j^2}=\sqrt{AB}.
$$
Hence
$$
\mu_{k+1}\leq \sqrt{\dfrac{A}{B}}.
$$
For any sequence $\{x_j\}_{j=k+1}^{\infty}$ with $\sum_{j=k+1}^{\infty}\mu_j^2x_j^2=A $ and
$\sum_{j=k+1}^{\infty}x_j^2=B $, we consider the following function
$$
F(x_j)= \sum_{j=k+1}^{\infty}\mu_jx_j^2+\lambda(\sum_{j=k+1}^{\infty}\mu_j^2x_j^2-A) +\mu(\sum_{j=k+1}^{\infty}x_j^2-B),
$$
where $\lambda$ and $\mu$ are Lagrange multipliers. Thus, the maximum $f_{max}$ of the function $f= \sum_{j=k+1}^{\infty}\mu_jx_j^2 $
is attained  at critical points of $F$.  If  $\{c_j\}_{j=k+1}^{\infty}$ is a critical point of $F$, for any sequence $\{b_j\}_{j=k+1}^{\infty}$,
we have
$$
\dfrac{dF(c_j+tb_j)}{dt}\biggl|_{t=0}=
2\sum_{j=k+1}^{\infty}\mu_jc_jb_j+2\lambda\sum_{j=k+1}^{\infty}\mu_j^2c_jb_j +2\mu\sum_{j=k+1}^{\infty}c_jb_j=0.
$$
By taking
$$
b_j=\begin{cases} 1 &j=p,\\
                  0 & j\neq p,
                  \end{cases}
                   $$
we have
$$
(\mu_p+\lambda\mu_p^2+\mu)c_p=0.
$$
Since $\mu_p+\lambda\mu_p^2+\mu=0$ is a quadratic equation of $\mu_p$, if $\mu_p+\lambda\mu_p^2+\mu\neq 0$, we have
$c_p=0$. Let $\mu_r$ and $\mu_s$, $r<s$, be solutions of $\mu_p+\lambda\mu_p^2+\mu=0$ with multiplicity $r_0+1$ and $s_0+1$, respectively,
that is,
$$
\mu_r=\mu_{r+1}=\cdots=\mu_{r+r_0} \quad \mu_s=\mu_{s+1}=\cdots=\mu_{s+s_0}.
$$
Therefore, we have
\begin{equation}
\begin{aligned}
&A=\mu_r^2(c_r^2+c_{r+1}^2+\cdots+c_{r+r_0}^2)+\mu_s^2(c_s^2+c_{s+1}^2+\cdots+c_{s+s_0}^2),\\
&B=(c_r^2+c_{r+1}^2+\cdots+c_{r+r_0}^2)+(c_s^2+c_{s+1}^2+\cdots+c_{s+s_0}^2),\\
&f_{max}=\mu_r(c_r^2+c_{r+1}^2+\cdots+c_{r+r_0}^2)+\mu_s(c_s^2+c_{s+1}^2+\cdots+c_{s+s_0}^2)\\
\end{aligned}
\end{equation}
Hence, we get
$$
f_{max}=\dfrac{A+\mu_r\mu_sB}{\mu_r+\mu_s}.
$$
Since $f_{max}\leq \sqrt{AB}$ from the Cauchy-Schwarz inequality, we have
$$
f_{max}=\dfrac{A+\mu_r\mu_sB}{\mu_r+\mu_s}\leq \sqrt{AB}.
$$
Thus, we obtain
$$
(\sqrt{\dfrac AB}-\mu_r)(\sqrt{\frac AB}-\mu_s)\leq 0,
$$
that is, we have
$$
\sqrt{\dfrac AB}-\mu_r\geq 0, \quad \sqrt{\frac AB}-\mu_s\leq 0
$$
because of  $\mu_r\leq \mu_s$.
Since $\sqrt{\dfrac AB}-\mu_r\geq 0$, we know that $G(t)=\dfrac{A+\mu_rtB}{\mu_r+t}$ is a decreasing function of $t$.
Hence, we have
$$
f_{max}\leq \dfrac{A+\mu_r\mu_{k+2}B}{\mu_r+\mu_{k+2}}.
$$
If $\mu_{k+2}\geq \sqrt{\frac AB}$, we have $\mu_r=\mu_{k+1}$ because of $r<s$ and $\mu_r\leq \sqrt{\dfrac AB}$, that is
$$
f_{max}\leq \dfrac{A+\mu_{k+1}\mu_{k+2}B}{\mu_{k+1}+\mu_{k+2}}.
$$
If $\mu_{k+2}\leq \sqrt{\frac AB}$,  we know that $G(t)=\dfrac{A+\mu_{k+2}tB}{\mu_{k+2}+t}$ is a decreasing function of $t$.
Hence, we have
$$
f_{max}\leq \dfrac{A+\mu_r\mu_{k+2}B}{\mu_r+\mu_{k+2}}\leq \dfrac{A+\mu_{k+1}\mu_{k+2}B}{\mu_{k+1}+\mu_{k+2}}.
$$
It completes the proof of the lemma.

\end{proof}

\end{document}